\documentclass{amsart}

\usepackage{times}      % use Times fonts for text
\usepackage{latexsym}   % use all LaTeX fonts
\usepackage{amssymb}    % use all AMS fonts
\usepackage{amsmath}    % include some AMS-LaTeX functionality
\usepackage{amsbsy}
\usepackage{amsthm}
\usepackage{amsgen}
\usepackage{amsfonts}
\usepackage{array}
\usepackage{epsfig}
\usepackage{graphicx}
\usepackage{psfrag}
\usepackage{hyperref}
\usepackage{pstricks}
\usepackage{psfrag}

\newtheorem{theorem}{Theorem}[section]

\newtheorem*{thma}{Theorem A}
\newtheorem*{thmb}{Theorem B}

\newtheorem{corollary}[theorem]{Corollary} 
 
\newtheorem*{msrconj}{Maximal Symmetry Rank Conjecture}
\newtheorem{lemma}[theorem]{Lemma}

\newtheorem{proof-ref}{Proof(Proposition{P:5t2})}

\def\ccc{\mathbb{C}}

\DeclareMathOperator{\Fix}{Fix}
\def\bdm{\begin{displaymath}}
\def\edm{\end{displaymath}}
\def\beq{\begin{equation}}
\def\eeq{\end{equation}}
\def\bes{\begin{equation*}}
\def\ees{\end{equation*}}
\def\epcm{\end{picture}\end{center}\end{minipage}}
\def\bpcm{\begin{minipage}{80pt}\begin{center}\begin{picture}}

\def\t2{T^2}

\def\f4{F_4}
\def\g2{G_2}

\def\p2{\frac{\pi}{2}}

\def\Fix{\textrm{Fix}}

\theoremstyle{remark}
\newtheorem{rem}[theorem]{Remark}

\theoremstyle{definition}

\begin{document}

%======================
%	BEGIN FRONT MATTER
%======================

%---------------------------------------
%	TOP MATTER
%---------------------------------------

% TITLE
\title[Low-dimensional manifolds with non-negative curvature and maximal symmetry rank]{Low-dimensional manifolds with non-negative curvature and maximal symmetry rank}

% AUTHOR 1
\author[Galaz-Garcia]{Fernando Galaz-Garcia}
\address[Galaz-Garcia]{Mathematisches Institut, WWU M\"unster, GERMANY}
%\curraddr{}
\email[]{f.galaz-garcia@uni-muenster.de}

% AUTHOR 2
\author[Searle]{Catherine Searle$^{*}$}
\address[Searle]{Institute of Mathematics, Universidad Nacional Autonoma 
de Mexico, Cuernavaca, Morelos, MEXICO}
\email[]{csearle@matcuer.unam.mx}

% MATH SUBJECT CLASSIFICATION

% THANKS
\thanks{$^*$
The second named author was supported in part by 
CONACYT Project \#SEP-CO1-46274, CONACYT Project \#SEP-82471 and UNAM DGAPA project IN-115408.}
\thanks {The authors thank the American Institute of Mathematics (AIM)  for its support during a workshop where the work on this paper was initiated.}

% DATE
\date{\today}

%-------------------------------------------
%		ABSTRACT
%-------------------------------------------
\begin{abstract}
We classify closed, simply connected $n$-manifolds of non-negative sectional curvature admitting an isometric torus  action of maximal symmetry rank in dimensions $2\leq n\leq 6$. In dimensions $3k$, $k=1,2$ there is only one such manifold and it is diffeomorphic to the product of $k$ copies of the $3$-sphere.
\end{abstract}

\maketitle

%======================
%	END FRONT MATTER
%======================

%======================
%	BEGIN MAIN MATTER
%======================

%--------------------------------------------------------
%	Section: Introduction and main results
%--------------------------------------------------------

\section{Introduction and main results}

% WHAT? WHY?

As in the case of closed Riemannian manifolds of positive sectional curvature, it is of interest to classify closed, non-negatively curved Riemannian manifolds whose isometry groups are large. One possible measure for the size of the isometry group of a Riemannian manifold $M$ is its \emph{symmetry rank}, denoted by $\mathrm{symrank}(M)$ and  defined to be the rank of the isometry group of $M$. This invariant was introduced by Grove and Searle in \cite{GS}, where it was shown that the symmetry rank  of a closed, positively curved Riemannian $n$-manifold is bounded above by $[\frac{n+1}{2}]$ and any such manifold with maximal symmetry rank  is diffeomorphic to a sphere, a real or complex projective space, or a lens space.

Given a Riemannian manifold  $M$, $\mathrm{symrank}(M)=k$ if and only if a $k$-torus $T^k$ acts isometrically and (almost) effectively on $M$ and there is no such action on $M$ by a $(k+1)$-torus. In  non-negative curvature, a flat $n$-torus has maximal symmetry rank $n$. If we also require simple connectivity, the Riemannian product of  $k$ copies of the round $3$-sphere $S^3$  has non-negative curvature and  admits an (almost) effective isometric $T^{2k}$ action. A similar statement is true in dimensions $3k+1$ and $3k +2$ by substituting one of the $3$-spheres for a round $S^4$ or simply augmenting the product with a round $S^2$, respectively.   Moreover, each one of these simply connected Riemannian products of spheres, with total dimension $n=3k,3k+1$ or $3k+2$,  has maximal symmetry rank $[\frac{2n}{3}]$ among  Riemannian products of round spheres, $S^{k_1}\times\cdots\times S^{k_l}$, with corresponding total dimension $k_1+\cdots+k_l=n$. Motivated by these facts, one may make the following conjecture:

% CONJ: MSR

\begin{msrconj}\label{c:c}
Given a closed, simply connected, non-negatively curved Riemannian manifold $M^n$, the following  are true:
\begin{itemize}
	\item[(1)] $\mathrm{symrank}(M^n)\leq [\frac{2n}{3}]$.
	
	\item[(2)] If $n=3k$, $k\geq 1$, and $M^{3k}$ admits an (almost) effective isometric $T^{2k}$ action,  then $M^{3k}$ is diffeomorphic to the product of $k$ copies of $S^3$.
\end{itemize}
\end{msrconj}

We will verify (1) of the conjecture in dimensions $2\leq n\leq 9$. We also classify, up to diffeomorphism, closed, simply connected, non-negatively curved Riemannian $n$-manifolds with  maximal symmetry rank in dimensions $2\leq n\leq 6$, verifying (2) of the conjecture for $k=1$ and $2$. We point out that  in contrast to the  positively curved case, where there is rigidity for odd-dimensional  simply connected manifolds
 (cf. \cite{GS}), in the non-negatively curved case one can only expect rigidity in dimensions $3k$, $k\geq 1$. 

% HOW?

Recall that the \emph{cohomogeneity} of an action of a group $G$ on a manifold $M$ is the dimension of the orbit space $M/G$. In terms of this invariant, the Maximal Symmetry Rank Conjecture states that for a closed, simply connected, non-negatively curved Riemannian manifold of dimension $3k$, $3k-1$ or $3k-2$, $k\geq 1$, the maximal symmetry rank corresponds to an (almost) effective, isometric torus action of cohomogeneity $k$. In particular, in dimensions $2$ and $3$, maximal symmetry rank corresponds to  an (almost) effective isometric cohomogeneity one torus action. The Maximal Symmetry Rank Conjecture holds in this case without any curvature assumptions, as a consequence of Mostert's \cite{M} and Neumann's \cite{N} classification of smooth  $2$- and $3$-manifolds of cohomogeneity one.  Both  $S^2$ and $S^3$, the simply connected manifolds with maximal symmetry rank in dimensions $2$ and $3$, equipped with the standard Riemannian metric of constant positive curvature, admit effective, isometric torus actions of cohomogeneity one. 

In dimensions $4$, $5$ and $6$, maximal symmetry rank corresponds to an (almost) effective, isometric  torus action of cohomogeneity two. In a curvature-free setting, smooth torus actions of cohomogeneity two on closed, smooth manifolds were studied by Orlik and Raymond \cite{OR, OR2} and Pao \cite{Pa1,Pa2}, in dimension  $4$,  and by Oh \cite{Oh,Oh2} in dimensions $5$ and $6$. Their work includes topological classification results in the simply connected case, given in terms of orbit space data. In our case, the presence of non-negative curvature imposes restrictions on the orbit space structure of the torus action. This observation allows us to apply the classification results of Orlik and Raymond and Oh to obtain our first main result, which we prove in section~\ref{S:Proof_thm_A}.

% THM A

\begin{thma}
\label{T:A}
Let $M^n$ be a closed, simply connected Riemannian $n$-manifold with nonnegative curvature and an (almost) effective, isometric  torus action of cohomogeneity two.
	\begin{itemize}
		\item[(1)] If $n=4$, then $M^4$ is diffeomorphic to $S^4$, $\ccc P^2$, $S^2\times S^2$ or $\ccc P^2\#\pm \ccc P^2$.
		\item[(2)] If $n=5$, then $M^5$ is diffeomorphic to $S^5$, $S^2\times S^3$ or $S^2\tilde{\times}S^3$, the non-trivial $S^3$-bundle over $S^2$.
		\item[(3)] If $n=6$, then $M^6$ is diffeomorphic to $S^3\times S^3$.
	\end{itemize}
\end{thma}

It is well-known that the manifolds listed in Theorem~A  admit a metric of non-negative curvature (cf. \cite{C,T} for $\ccc P^2\#\pm \ccc P^2$ and \cite{Pav} for $S^2\tilde{\times} S^3$). We then show in section~\ref{S:Proof_thm_B} that the conjectured bound on the symmetry rank holds through dimension $9$, obtaining our second main result.

% THM B

\begin{thmb}\label{T:B} Let $T^k$ act (almost) effectively and isometrically on a closed, simply connected Riemannian $n$-manifold with non-negative curvature. If $n\leq 9$, then $k\leq  [\frac{2n}{3}]$.
\end{thmb}

It follows from Theorem~B that the closed, simply connected Riemannian $n$-manifolds with nonnegative curvature and maximal symmetry rank, in dimensions $2\leq n\leq 6$, are $S^2$, $S^3$ and those listed in Theorem~A.

We remark that in dimensions $n\geq 7$ there are no general topological classification results  to be found in the literature on smooth $n$-manifolds with smooth torus actions of cohomogeneity $n-[\frac{2n}{3}]$, corresponding to the (conjectured) case of maximal symmetry rank  for a closed, simply connected, non-negatively curved Riemannian $n$-manifold. 

%---------------------------------------------------------------
%	Section: Proof of Theorem~A (Classification)
%---------------------------------------------------------------

\section{Proof of Theorem~A}
\label{S:Proof_thm_A}

%------------------------------------------------------------------
% Subsection: Torus actions of cohomogeneity two
%------------------------------------------------------------------
We prove (1)--(3) of Theorem~A in subsections~\ref{SS:D4} and \ref{SS:D5}. We begin by recalling  some  facts about smooth cohomogeneity two torus actions on smooth manifolds.

\subsection{Torus actions of cohomogeneity two} 
\label{SS:cohom2_torus}We first fix some terminology and notation. Given a compact Lie group $G$ acting (on the left) on a smooth manifold $M$,   we denote by $G_x=\{\, g\in G : gx=x\, \}$ the \emph{isotropy group} at $x\in M$ and by $Gx=\{\, gx : g\in G\, \}\simeq G/G_x$ the \emph{orbit} of $x$. The \emph{ineffective kernel} of the action is the subgroup $K=\cap_{x\in M}G_x$. We say that $G$ acts \emph{effectively} on $M$ if $K$ is trivial. The action is called \emph{almost effective} if $K$ is finite. We will denote the fixed point set $\{\, x\in M : gx=x \text{ for all } g\in G \, \}$ of this action by $\Fix(M , G )$. We will let $Q$ denote the union of singular orbits. Given a subset $A\subset M$, we will denote its image in $M/G$ under the orbit projection map by $A^*$. When convenient, we will also denote the orbit space $M/G$ by $M^*$.

\begin{rem} We will henceforth assume all manifolds to be smooth, and we will only consider smooth,  (almost) effective actions. 
\end{rem}

Let $M^{n+2}$  be a closed, simply connected manifold with a cohomogeneity two action of a compact connected Lie group $G$. It is well-known (see, for example, \cite{Br} Chapter IV) that in the presence of singular orbits, the orbit space $M^*$ of the action is homeomorphic to a $2$-disk $D^2$ with boundary $Q^*$, the projection of  the singular orbits, and the interior points of $M^*\simeq D^2$ correspond to the principal orbits. In the particular case where $G=T^n$, $n\geq 2$, the orbit space structure  was analyzed in \cite{OR, KMP} (see also \cite{Oh}). Here, the only possible non-trivial isotropy groups are $T^1$ and $T^2$. The boundary circle, $Q^*$,  is a union of arcs, and the interior of each arc corresponds to orbits with isotropy $T^1$, while the endpoints of each arc correspond to orbits with isotropy $T^2$ (see Figure~\ref{F:orbit_space}). Moreover, 
there must be at least $n$ orbits with isotropy $T^2$.  For proofs of these facts, please refer
 to \cite{OR, KMP,Oh}.
% FIG: ORBIT SPACE

\begin{figure}
\psfrag{A}{Isotropy $T^2$}
\psfrag{B}{Isotropy $T^1$}
\psfrag{C}{Principal orbits}
\includegraphics{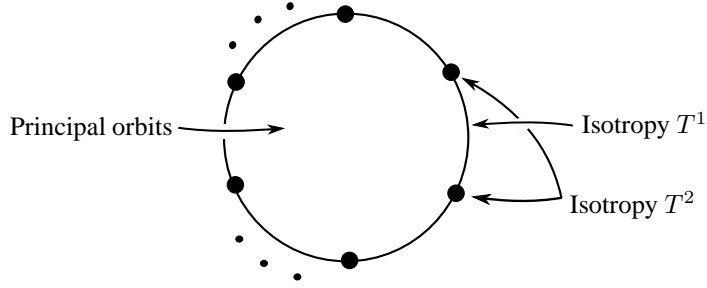}
\caption{Orbit space structure}\label{F:orbit_space}
\end{figure}

The following lemma allows us to limit the total number of orbits with isotropy $T^2$  in the presence of non-negative curvature.

% LEM: BOUND ON ORBITS WITH T2 ISOTROPY

\begin{lemma}\label{l:4T2} Let $M^{n+2}$, $n\geq 2$, be a closed, simply connected manifold with non-negative curvature and an isometric  $T^{n}$-action. Then there are at most four orbits with isotropy $T^2$. 
\end{lemma} 

\begin{proof}
Recall, from the discussion above, that  $M^*$ is homeomorphic to a $2$-disk whose boundary consists of singular orbits. Since $n\geq 2$, at least two points in $\partial M^*$  correspond to orbits with isotropy $T^2$. Moreover, $M^*$ is a non-negatively curved Alexandrov space and  the space of directions at the points with isotropy $T^2$ is isometric to a closed interval of length $\pi/2$. Suppose now that there are five points $p^*_0,\ldots,p^*_4$  with isotropy $T^2$ in $\partial M^*$, and assume that the points $p^*_i$, $p^*_{i+1}$, $0\leq i\leq 4$  are adjacent, with indices considered modulo $5$.  We obtain a triangulation of $M^*$ consisting of three triangles by taking minimal geodesics from $p^*_0$ to $p^*_2$ and $p^*_3$ (see Figure~\ref{F:triangulation}).  By Toponogov's theorem, the sum of the angles of these three triangles is bounded below by $3\pi$. On the other hand, summing the vertex angles at the points $p^*_i$, $0\leq i\leq 4$, we see that the sum of the angles of these three triangles is bounded above by $5\pi/2$, giving us a contradiction. 
\end{proof}

The following corollary is a consequence of Lemma~\ref{l:4T2} and the fact, mentioned above, that a smooth action of $T^n$ on a closed, simply connected manifold $M^{n+2}$, $n\geq 2$, must have at least $n$ orbits with isotropy $T^2$.

\begin{corollary}\label{c:nocohom2} No closed, simply connected Riemannian manifold with non-negative curvature and dimension $n\geq 7$ admits an isometric cohomogeneity two torus action. 
\end{corollary}

% FIG: TRIANGULATION

\begin{figure}
\psfrag{A}{$p_0^*$}
\psfrag{B}{$p_1^*$}
\psfrag{C}{$p_2^*$}
\psfrag{D}{$p_3^*$}
\psfrag{E}{$p_4^*$}
\includegraphics{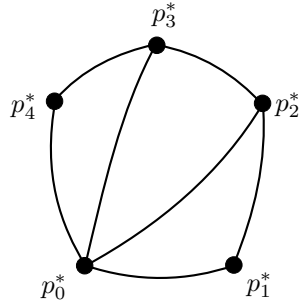}
\caption{Orbit space triangulation}\label{F:triangulation}
\end{figure}

%------------------------------------------------------
%       Subsection: Dimension 4, T2 action
%-----------------------------------------------------

\subsection{Dimension $4$}
\label{SS:D4} We recall the classification of closed, simply connected $4$-manifolds with a smooth $T^2$-action.  

% THM: 4-MANIFOLDS WITH A T2 ACTION

\begin{theorem}[Orlik and Raymond \cite{OR2}]\label{T:OR}  A closed,  simply connected $4$-manifold with a $T^2$ action is equivariantly diffeomorphic to a connected sum of $S^4$, $\pm\ccc P^2$ and $S^2\times S^2$.
\end{theorem}

Let $M^4$ be a closed, simply connected, non-negatively curved $4$-manifold with an isometric $T^2$ action and let $\chi(M^4)$ be its Euler characteristic. Given a circle subgroup $T^1\subset T^2$,  it is well-known that $\chi(M^4)= \chi(\Fix((M^4, T^1))=\chi(\Fix(M^4,T^2))$, which is the number of orbits with $T^2$ isotropy. It follows from Lemma~\ref{l:4T2} and Poincar\'e duality that $2\leq \chi(M^4)\leq 4$. Combining this with Theorem~\ref{T:OR} yields (1) of Theorem~A.

%--------------------------------------------------------------------------------------
% 	Subsection: Dimension 5, T3 action; Dimension 6, T4 action.
%--------------------------------------------------------------------------------------

\subsection{Dimensions $5$ and $6$}
\label{SS:D5}

Let $M^{n+2}$ be a closed, simply connected manifold with non-negative curvature and an isometric $T^n$-action. Let $k$ be the number of orbits with isotropy $T^2$. It follows from Lemma~\ref{l:4T2} and the discussion in subsection~\ref{SS:cohom2_torus} that  $3\leq k \leq 4$ for $n=3$, and $k=4$ for $n=4$.  Parts (2) and (3) of Theorem~A then follow from Oh's classification of closed,  simply connected $5$- and $6$-manifolds with cohomogeneity two torus actions, which we now recall. We denote the second Stiefel-Whitney class of a manifold $M$ by $w_2(M)$.

% THM: OH'S CLASSIFICATION OF  5- AND 6-MANIFOLDS WITH COHOM 2 TORUS ACTIONS

\begin{theorem}[Oh \cite{Oh,Oh2}]\label{Thm:Oh} Let $M^{n+2}$ be a closed, simply connected manifold with a $T^{n}$ action and $k$ orbits with isotropy $T^2$. 

\begin{itemize}
	\item[(1)] If $n=3$, then $M^5$ is equivariantly diffeomorphic to one of the following:
	\begin{itemize}
		\item[] $S^5$, if $k=3$;\\
		\item[] $\#(k-3)(S^2\times S^3)$, if $w_2(M^5)=0$;\\
		\item[] $(S^2\tilde{\times}S^3)\#(k-4)(S^2\times S^3)$, if $w_2(M^5)\neq 0$.\\
	\end{itemize}
	\item[(2)] If $n=4$, then $M^6$ is diffeomorphic to one of the following:
	\begin{itemize}
		\item[] $\#(k-4)(S^2\times S^4)\#(k-3)(S^3\times S^3)$, if $w_2(M^6)=0$;\\
		\item[] $(S^2\tilde{\times}S^4)\#(k-5)(S^2\times S^4)\#(k-3)(S^3\times S^3)$, if $w_2(M^6)\neq 0$.\\
	\end{itemize}
\end{itemize}
\end{theorem}

%-----------------------------------------------------------------------------
%	Section: Proof of Theorem B (Maximal symmetry rank)
%-----------------------------------------------------------------------------

\section{Proof of Theorem~B}
\label{S:Proof_thm_B}

The arguments in dimension $n\leq 6$ are topological and require no curvature assumptions. A smooth $2$-manifold with a $T^2$ action must be homogeneous, and hence cannot be simply connected. Similarly, a smooth $3$-manifold with a $T^3$ action cannot be simply connected. 

For $n=3,4$ or $5$, a $T^n$-action on a smooth manifold $M^{n+1}$ is of  cohomogeneity one. It follows  in dimension $4$ from  Parker \cite{Pa}, and in higher dimensions from Pak \cite{P}, that
$M^{n+1}$ is homeomorphic to the product of a torus $T^{n-2}$ with a closed $3$-manifold of cohomogeneity one. Hence $M^n$ cannot be simply connected.

In dimensions $7$, $8$ and $9$, Corollary~\ref{c:nocohom2} rules out   the existence of a cohomogeneity two torus action.

%======================
%	END MAIN MATTER
%======================

%======================
%	BEGIN BACK MATTER
%======================

%-------------------------------------------
%        Bibliography
%-------------------------------------------

%======================
%	END BACK MATTER
%======================

\end{document}